\documentclass{article}

\usepackage{multicol}
\usepackage{graphicx}
\usepackage{amsmath}
\usepackage{amsthm}
\usepackage{amssymb}
\usepackage{amsfonts}
\usepackage{tikz-cd}
\usepackage[all]{xy}
\usepackage{enumerate}
\usepackage{multicol}
\usepackage[colorlinks=true,linkcolor=blue]{hyperref}

\newtheorem{theorem}{Theorem}[section]
\newtheorem{lemma}[theorem]{Lemma}
\newtheorem{proposition}[theorem]{Proposition}
\newtheorem{corollary}[theorem]{Corollary}
\newtheorem{question}[theorem]{Question}

\newtheorem{notation}[theorem]{Notation}

\newtheorem{remark}[theorem]{Remark}

\newcommand{\Z}{\mathbb{Z}}
\newcommand{\Q}{\mathbb{Q}}
\newcommand{\R}{\mathbb{R}}
\newcommand{\C}{\mathbb{C}}

\newcommand{\F}{\mathbb{F}}

\newcommand{\Ocal}{\mathcal{O}}

\newcommand{\JR}{{\rm JR }}
\newcommand{\disc}{{\rm disc\: }}
\newcommand{\Norm}{{\rm Norm }}
\newcommand{\Gal}{{\rm Gal }}

\begin{document}

\title{Julia Robinson numbers and arithmetical dynamic of quadratic polynomials}
\author{Marianela Castillo Fern\'andez, Xavier Vidaux and Carlos R. Videla}

\maketitle

\begin{abstract}
For rings $\Ocal_K$ of totally real algebraic integers, J. Robinson defined a set which is always $\{+\infty\}$ or of the form $[\lambda,+\infty)$ or $(\lambda,+\infty)$ for some real number $\lambda\ge4$. All known examples give either $\{+\infty\}$ or $[4,+\infty)$. In this paper, we construct infinitely many fields such that the set is an interval, but not equal to $[4,+\infty)$.\footnote{The three authors have been partially supported by the first author Fondecyt research projects 1130134 and 1170315, Chile. This work was partially financed by the first author Conicyt fellowship ``Beca Doctorado Nacional'' and by the Universidad de Concepci\'on, Chile. Part of this work was done when the first author was visiting the third author in Calgary, in January-April 2015. She thanks the Department of Mathematics and Computation of Mount Royal University for their hospitality during her stay. Part of this work was done while the third author was on a sabbatical leave during January-June 2017.}
\end{abstract}

MSC 2010: Primary: 11R04, 11R32, 11R80, 37P05
Secondary: 11R09, 11R11

Keywords: Julia Robinson number, totally real towers, totally real algebraic numbers, wreath product, iterates of quadratic polynomials

%\tableofcontents

%%%%%%%%%%%%%%%%%%%%%%%%%
%%%%%%%%%%%%%%%%%%%%%%%%%

\section{Introduction}

Motivated by a problem in Logic, for any given field $K$ of totally real algebraic numbers, Julia Robinson \cite{Rob62} considered the following set $A(\Ocal_K)$:
$$
\{t\in\R\cup\{+\infty\}\colon \textrm{ there are infinitely many $r\in\Ocal_K$ such that $0\ll r\ll t$}\}, 
$$
where ``$0\ll r\ll t$'' means that every conjugate of $r$ lies strictly between $0$ and $t$, and $\Ocal_K$ is the ring of integers of $K$. The set $A(\Ocal_K)$ is either the set $\{+\infty\}$, or an interval of the form $[\lambda,+\infty)$ or $(\lambda,+\infty)$. We define the $\JR$ number of $\Ocal_K$ to be $+\infty$ when $A(\Ocal_K)=\{+\infty\}$, and $\lambda$ when it is an interval. A consequence of a theorem of Kronecker is that we always have $\lambda\ge4$ --- see \cite{Rob62}, or \cite{JyV} for a more detailed account. In all her examples, $A(\Ocal_K)$ is either $\{+\infty\}$ or $[4,+\infty)$. She asks whether there is any $K$ such that $A(\Ocal_K)$ is an open interval. 

The $\JR$ number of a ring is relevant for decision problems in Logic, as discovered by Julia Robinson, and it is also connected to the Northcott property of sets of algebraic numbers --- see \cite{VV16} and \cite{Widmer16}. Note that in \cite{VyV}, it is proved that there are subrings $R$ of rings of the form $\Ocal_K$ such that $A(R)$ is an open interval. It is not known whether any of these rings is an $\Ocal_K$. 

An immediate corollary of our main result is the following. 

%%%%
\begin{theorem}\label{Main}
There are infinitely many fields $K$ with $A(\Ocal_K)$ distinct from $\{+\infty\}$ and $[4,+\infty)$. 
\end{theorem}
%%%%

Our fields $K$ are the ones that were already considered in \cite{VyV}, namely, fields of nested square roots defined in the following way. Let $\nu$ be a non-square integer $\ge4$. Let $x_1=\sqrt{\nu}$, and for each $n\ge1$, $x_{n+1}=\sqrt{\nu+x_n}$. Note that for each $n$, the field $K_n=\Q(x_n)$ is totally real, and that since $\nu$ is not a square, the tower increases at each step --- apply \cite[Cor. 1.3]{Stoll} to the iterated of $f(t)=t^2-\nu$.
Write
$$
\Ocal^{\nu}=\cup_n \Ocal_{K_n}. 
$$
Notice that $\Ocal^\nu$ is the ring of integers of $\bigcup_n K_n$. 

We prove:

%%%%
\begin{theorem}\label{main prop}
Suppose that $\nu=2^{2m}\mu$, with $m\ge1$, $\mu\ge3$ odd and not a quadratic residue modulo any Fermat prime greater than $3$. The $\JR$ number of $\Ocal^{\nu}$ is either strictly between 4 and $+\infty$, or it is $4$ and it is not a minimum.
\end{theorem}
%%%%

In Section \ref{sec fermat primes} we prove that $3$ and $7$ are non-squares modulo any Fermat prime greater than $3$. So for example, for any odd integer $k$ and for any $m\ge1$, $\nu=2^{2m}k^2\cdot3$ and $\nu=2^{2m}k^2\cdot7$ satisfy the hypothesis of Theorem \ref{main prop}.

%%%%%%%%%%%%%%%%
%%%%%%%%%%%%%%%%
%%%%%%%%%%%%%%%%
\section{Sketch of proof}

The fact that the $\JR$ number is not $\{+\infty\}$ is an immediate consequence of the fact that $\Ocal^{\nu}$ has a subring with $\JR$ number not $\{+\infty\}$ --- the subring in question is $\bigcup_n\Z[x_n]$ and its $\JR$ number is $\lceil\alpha\rceil+\alpha$, where $\alpha=(1+\sqrt{1+4\nu})/2$. This is proven in \cite[Thm. 1.4]{VyV}. The hypothesis of this theorem that $\nu$ must be congruent to $2$ or $3$ modulo $4$ is not satisfied for our choice of $\nu$, but this hypothesis was only there to ensure that the tower increases at each step. 

If $m$ is an integer, we write $\zeta_m$ for a primitive $m$-th root of unity. 

The $\JR$ number of $\Ocal^{\nu}$ is $4$ and is a minimum if and only if, in $\Ocal^{\nu}$ there exist infinitely many numbers of the form 
$$
\zeta_m^j+\zeta_m^{-j}=2\cos\left(\frac{2\pi j}{m}\right),
$$
with $j=1,\dots,m-1$, if and only if in $\Ocal^{\nu}$ there exist infinitely many numbers of the form 
$$
\zeta_m+\zeta_m^{-1}=2\cos\left(\frac{2\pi }{m}\right). 
$$
The first equivalence is a consequence of theorem of Kronecker, see \cite[Thm. 2.5]{Nark}. 

Since the fraction field of $\Ocal^{\nu}$ is a $2$-tower, we have the following equivalence for each $m$: $\zeta_m+\zeta_m^{-1}\in \Ocal^{\nu}$ if and only if $\zeta_m+\zeta_m^{-1}$ is constructible with ruler and compass, if and only if $m=2^d p_1 \dots p_k$, where $d\geq 0$ and $p_i$ are distinct Fermat Primes (by Gauss-Wantzel Theorem).  Thus, the strategy consists in finding $\nu$ such that $\Ocal^{\nu}$ has only finitely many numbers $\zeta_m+\zeta_m^{-1}$ with $m$ of the form $2^d p_1 \dots p_k$. 

The proof is then done in two steps. In Section \ref{proof 3}, we will prove the following proposition. 

\begin{proposition}\label{mainthqqq}
Assume that $K_n$ has degree $2^n$ over $\Q$. Suppose that 
\begin{enumerate}
	\item for every Fermat prime $p>3$, $\nu$ is not a square modulo $p$, and
	\item $\sqrt{2}$ is not in $\Ocal^{\nu}$.
\end{enumerate}
The $\JR$ number of $\Ocal^{\nu}$ is either strictly between 4 and $+\infty$, or it is $4$ and it is not a minimum.
\end{proposition}

In Section \ref{sec galois group} we prove that if $\nu=2^{2m}\mu$, with $m\geq 1$ and $\mu\geq 3$ odd and square-free, then $\sqrt{2}$ is not in $\Ocal^{\nu}$.

Putting everything together, this proves Theorem \ref{main prop}. Note that if there are only finitely many Fermat primes, then item 1 of Proposition \ref{mainthqqq} is not relevant for our purposes, because they would contribute only to finitely many elements of the form $\zeta_m+\zeta_m^{-1}$.

%%%%%%%%%%%%%%%%%%%%%%%%%%%%%%%%%%%%%%%%%%%%%%%%%%%%%%%%%%%%%%%%%%%%%%%%%%%%%%%%%%%%%%%%%%%
\section{Discriminant of $x_n$.}\label{disc xn}

For each $n$, let $P_n(t)=f^{\circ n}(t)$, where $f(t)=t^2-\nu$. By \cite[Corollary 1.3]{Stoll}, each polynomial $P_n$ is the minimal polynomial of $x_n$. In this section we prove the following proposition. 

%%%%
\begin{proposition}\label{marcus} Assume that $K_n$ has degree $2^n$ over $\Q$. We have
$$
\disc(x_1)=4\nu,
$$
and for $n\geq 2$ we have
$$
\disc(x_n)=(\disc(x_{n-1}))^2\cdot 2^{2^n} P_n(0).
$$
\end{proposition}
%%%%

The field $K_n$ has basis
$$
B_n:=\{1,x_n,x_n^2,\dots,x_n^{2^n-1}\}
$$
over $\Q$. Note that the field extension $K_n\slash K_m$ has degree $2^{n-m}$. We will denote by $\disc^{n}_{n-1}(x_n)$ the discriminant of the basis $(1,x_n)$ from $K_n$ to $K_{n-1}$. Hence, for $n\geq 1$, we have
	$$
	\disc^n_{n-1}(x_n) =\left|
	\begin{array}{cc}
		1 & x_n\\
		1 & -x_n
	\end{array}\right|^2
	=4(x_n)^2=4(\nu+x_{n-1}).
	$$

\begin{notation}
 For $n\geq 1$, we denote by $N_n$ the norm from $K_n$ to $\Q$ of $\disc^{n+1}_{n}(x_{n+1})$, and by $N_0$ the discriminant of $x_1$ from $K_1$ to $\Q$.
\end{notation}

%%%%%%%%%%%%%%%%%%%%%%%
\begin{proposition}\label{propnorm}
We have
\begin{enumerate} 
\item $N_0=2^2\nu$, and 
\item $N_n=2^{2^{n+1}} P_{n+1}(0)$ for any $n\geq 1$. 
\end{enumerate}  
\end{proposition}

\begin{proof}
Item 1 is immediate from our above computation, so we prove item 2. Let $\ell_1=\nu^2$ and $\ell_n=((\ell_{n-1})-\nu)^2$ for $n\geq 2$.  Let $n\geq 1$. We have
$$
\begin{aligned}
N_n &=\Norm^{K_n}_{\Q} \left( \disc^{n+1}_{n}(x_{n+1})\right)\\
    &=\Norm^{K_n}_{\Q} ( 4(\nu+x_n))\\
		&=({2^2})^{2^n}\Norm^{K_n}_{\Q} (\nu+x_n)\\
		&=2^{2^{n+1}} \prod_{i=1}^{2^n} (\nu+x_n^{\sigma^{n}_{i}}),
\end{aligned}
$$
where the $\sigma^n_i$ are the $2^n$ embeddings from $K_n$ to $\C$.

\textbf{Fact.} \emph{For all $t\in \left\{0,\dots,n\right\}$ we have 
$$
\prod_{i=1}^{2^n}\left(\nu+x_n^{\sigma^{n-1}_{i}}\right)=\prod_{i=1}^{2^{n-t}}\left(\ell_t-(\nu+x_{n-t})^{\sigma_{i}^{n-t}}\right).
$$
}

We prove the fact by induction on $t$. Assume it is true for $t-1$, namely, 
$$
\prod_{i=1}^{2^n}\left(\nu+x_n^{\sigma^{n-1}_{i}}\right)=\prod_{i=1}^{2^{n-(t-1)}}\left(\ell_{t-1}-(\nu-x_{n-(t-1)})^{\sigma_{i}^{n-(t-1)}}\right),
$$
we have
$$
\begin{aligned}
\prod_{i=1}^{2^n}\left(\nu+x_n^{\sigma^{n-1}_{i}}\right)&=\prod_{i=1}^{2^{n-t+1}}\left(\ell_{t-1}-(\nu-x_{n-t+1)})^{\sigma_{i}^{n-t+1}}\right)\\
&=\prod_{i=1}^{2^{n-t+1}}\left(\left(\ell_{t-1}-\nu\right)+x_{n-t+1}^{\sigma_{i}^{n-t+1}}\right)\\
&=\prod_{i=1}^{2^{n-t}}\left(\left(\ell_{t-1}-\nu\right)-x_{n-t+1}^{\sigma_{i}^{n-t}}\right) \left(\left(\ell_{t-1}-\nu\right)+x_{n-t+1}^{\sigma_{i}^{n-t}}\right)\\
&=\prod_{i=1}^{2^{n-t}}\left(\left(\ell_{t-1}-\nu\right)^2-(x_{n-t+1}^{2})^{\sigma_{i}^{n-t}}\right)\\
&=\prod_{i=1}^{2^{n-t}}\left(\ell_t-(\nu+x_{n-t})^{\sigma_{i}^{n-t}}\right).
\end{aligned}
$$
This proves the fact. 

Hence, taking $t=n$ in the Fact above, we obtain
$$
\prod_{i=1}^{2^n}\left(\nu+x_n^{\sigma^{n-1}_{i}}\right)=\left(\ell_n-\nu\right)=P_{n+1}(0).
$$

\end{proof}
%%%%%%%

We need the following proposition --- see \cite[Chap. 2, Exercise 23, p. 43]{Marcus}. 

%%%%%%%%%%%%%%%%%%%%%
%\begin{proposition}[\cite{Nark}, p 150]\label{nark}
%If $K\subset L\subset M$, then
%$$
%\disc(M/K)=\disc(L/K)^{[M:L]}\cdot \Norm_{L/K}(\disc(M/L)).
%$$
%\end{proposition}

%%%
\begin{proposition}\label{marcus1}
Let $K\subset L\subset M$ be number fields, $\left[L \colon K\right]=n$, $\left[M \colon L\right]=m$, and let $\left\{\alpha_1,\dots, \alpha_n\right\}$ and $\left\{\beta_1,\dots, \beta_m\right\}$ be bases for $L$ over $K$ and $M$ over $L$, respectively.  We have
$$
\disc_K^M \left(\alpha_1\beta_1,\dots,\alpha_n\beta_m\right)
=\left(\disc_K^L(\alpha_1\dots,\alpha_n)\right)^m \cdot \Norm_K^L\left(\disc_L^M(\beta_1\dots,\beta_m)\right).
$$
\end{proposition}
%%%
 
Proposition \ref{marcus} follows from Propositions \ref{propnorm} and \ref{marcus1} in the following way. Take  
$$
K=\Q,\quad L=K_{n-1}\quad\textrm{and}\quad M=K_n.
$$
The degree of $L$ over $K$ is $2^{n-1}$ and $L$ has basis
$$
\left\{1,x_{n-1}, x^{2}_{n-1},\dots, x^{2^{n-1}-1}_{n-1}\right\}
$$ 
over $K$, while the degree of $M$ over $L$ is $2$ and $M$ has basis $\left\{1, x_n\right\}$ over $L$. The set $\left\{\alpha_1\beta_1,\dots, \alpha_n\beta_m\right\}$ in Proposition \ref{marcus1} corresponds to the set
$$
B'=\left\{1,x_{n-1}, x^{2}_{n-1},\dots, x^{2^{n-1}-1}_{n-1}, x_n,x_{n-1}x_n, x^{2}_{n-1}x_n,\dots, x^{2^{n-1}-1}_{n-1}x_n\right\}. 
$$
This set $B'$ is a basis for $M$ over $K$. Indeed, we have 
$$
|B'|=2\left(2^{n-1}-1\right)+2=2^n=|B_n|,
$$  
and since $x^2_n=\nu+x_{n-1}$, each element of $B_n$ can be written as a $\Z$-linear combination of elements of $B'$. Similarly, each element of $B'$ is a $\Z$-linear combination of elements of $B_n$. Since the base change matrices from $B_n$ to $B'$ and from $B'$ to $B_n$ have an integral determinant and because the discriminants are also integers, we deduce 
$$
\disc_K^M(B')=\disc_K^M(B_n)=\disc_K^M(x_n).
$$
One obtains the formula in Proposition \ref{marcus} by using in Proposition \ref{marcus1} the formulas from Proposition \ref{propnorm}.

%%%%%%%%%%%%%%%%
%%%%%%%%%%%%%%%%
%%%%%%%%%%%%%%%%
\section{Proof of Proposition \ref{mainthqqq}}\label{proof 3}

We will need the following proposition. 

%%%%
\begin{proposition}[Prop. 2.13, \cite{Nark}]\label{m} Let $\theta$ be an algebraic integer. We have 
$$
\disc(\theta)=m^2\disc(\Q(\theta)),
$$
where $m$ is the index in $\Ocal_{\Q(\theta)}$ of the $\Z$-module $\Z[\theta]$.
\end{proposition}
%%%%

The following remark shows that it is sufficient to consider $\zeta_m+\zeta_m^{-1}$ where $m\in \{2^d\colon d\geq 2\}\cup \{p\colon p \text{ is a Fermat prime}\}$.

\begin{remark}\label{remsinnombre}   Let $m_1$ and $m_2$ be positive coprime integers, and write $m=m_1m_2$.  The field $\Q(\zeta_{m_1 m_2})$ is the compositum of $\Q(\zeta_{m_1})$ and $\Q(\zeta_{m_2})$. 
\end{remark}

We need the following result.

\begin{proposition}\label{zeta}
\begin{enumerate}
	\item (\cite{Wash}, p. 15)\label{max subf} The field $\Q(\zeta_m+\zeta_m^{-1})$ is the maximal totally real subfield of $\Q(\zeta_m)$.  The extension $\Q(\zeta_m)/\Q(\zeta_m+\zeta_m^{-1})$ is of degree $2$.

	\item (\cite{Wash}, Ex. 2.1, p. 17)\label{sqr p} Let $p$ be a prime number. The field $\Q(\zeta_p)$ contains the field $\Q(\sqrt{p})$ if $p\equiv 1\pmod 4$ and contains $\Q(\sqrt{-p})$ if $p\equiv 3\pmod 4$.

	\item \label{ram} Let $K$ be a number field.  The number $p$ is ramified in $K$ if and only if $p$ divides $\disc K$.	
\end{enumerate}
\end{proposition}

We prove the following proposition.

%%%%
\begin{proposition}\label{fp zeta}
Let $p>3$ be a Fermat prime.  If $\zeta_p+\zeta_p^{-1} \in \Ocal^{\nu}$, then there exists $n\ge1$ such that $p$ divides $\disc K_n$.
\end{proposition}
\begin{proof}
Let $p=2^{2^m}+1>3$ be a Fermat prime. Note that, since $m\ge1$, $p$ is congruent to $1$ modulo $4$. Hence, by Proposition \ref{zeta}, we have 
$$
\Q(\sqrt{p})\subset\Q(\zeta_p+\zeta_p^{-1}),
$$
hence $\sqrt{p}\in \Ocal^{\nu}$ by hypothesis, so in particular $\sqrt{p}$ lies in $K_n$ for some $n\ge1$. Therefore, $p=(\sqrt{p})^2$ is ramified in $K_n$, so $p$ divides $\disc K_n$ by Proposition \ref{zeta}.
\end{proof}
%%%%

%%%%
\begin{lemma}\label{lem fp not square}
Let $p$ be an odd prime. If $p$ divides $\disc(x_n)$, then $p$ divides the product $P_1(0)\dots P_n(0)$.
\end{lemma}
\begin{proof}
By induction on $n$.  For $n=1$ we have $\disc (x_{1})=4\nu=-4P_1(0)$. 

If it is true for $n$, then it is true for $n+1$ by Proposition \ref{marcus}, since we have
$$
\disc(x_{n+1})=(\disc(x_{n}))^2\cdot 2^{2^{n+1}} P_{n+1}(0).
$$
\end{proof}
%%%%

%%%%
\begin{corollary}\label{cor fp not square}
Let $p$ be an odd prime. If $p$ divides $\disc(K_n)$ for some $n\ge1$, then $p$ divides the product $P_1(0)\dots P_n(0)$.
\end{corollary}
\begin{proof}
This is an immediate consequence of Lemma \ref{lem fp not square}, because we know by Proposition \ref{m} that 
the discriminant of $K_n$ divides the discriminant of $x_n$. 
\end{proof}
%%%%

%%%%
\begin{proposition}\label{fp not square}
Let $p>3$ be a Fermat prime. If $\nu$ is not a square modulo $p$ (so in particular $p$ does not divide $\nu$), then for each $n\geq 1$, $p$ does not divide $ \disc K_n$.
\end{proposition}
\begin{proof}
We prove by induction on $n$.  For $n=1$ we have that 
$$
\disc \Q(x_{1})=\left\{
\begin{aligned}
\nu   &, \text{ if } \nu\equiv 1 \mod 4\\
4\nu&, \text{ if } \nu\equiv 2, 3 \mod 4
\end{aligned}
\right.
$$
In both cases, since $p$ does not divide $\nu$, we have that $p$ does not divide $\disc \Q(x_1)$.  

Assume by contradiction that $p$ divides the discriminant of $K_n$, so that $p$ divides $P_j(0)$ for some $j\in\{1,\dots,n\}$ by Corollary \ref{cor fp not square}. If $j=1$, then $p$ divides $\nu$, which contradicts our hypothesis. Assume $j>1$. Recall that $P_n(t)=f^{\circ n}(t)$, where $f(t)=t^2-\nu$. Therefore, we have 
$$
P_j(0)=P_{j-1}(0)^2-\nu,
$$
which contradicts the hypothesis that $\nu$ is not a square modulo $p$. 
\end{proof}
%%%%

\begin{proof}[Proof of Proposition \ref{mainthqqq}]
We follow the strategy described in the introduction. Let $p$ be a Fermat Prime greater than $3$.  By Proposition \ref{fp not square}, if $\nu$ is not a square modulo $p$, then $p$ does not divide $\disc K_n$, so $\zeta_p+\zeta_p^{-1}$ does not lie in $\Ocal^{\nu}$ by Proposition \ref{fp zeta}.

Let $s_1=\sqrt{2}$ and $s_n=\sqrt{2+s_{n-1}}$. Since
$$
\zeta_{2^d}+\zeta_{2^d}^{-1}=
\begin{cases}
-2 &\text{ if } d=1\\
s_{d-1} &\text{ if } d\geq 2, 
\end{cases}
$$
and $\sqrt{2}$ is not in $\Ocal^{\nu}$ by hypothesis, $\zeta_{2^d}+\zeta_{2^d}^{-1}$ does not lie in $\Ocal^{\nu}$ for any $d\geq 2$. Remark \ref{remsinnombre} allows us to conclude.
\end{proof}

%%%%%%%%%%%%%%%%%%%%%%%%%%%%%%%%%%%%%%%%%%%%%%%%%%%%%%%%%%%%%%%%%%%%%%%%%%%%%%%%%%%%%%%%%%%%%%%%%%%%%%%%%%%%%%%%%%%%%%%%%%%%%%%%%%%%%%%%%%%%%
%%%%%%%%%%%%%%%%%%%%%%%%%%%%%%%%%%%%%%%%%%%%%%%%%%%%%%%%%%%%%%%%%%%%%%%%%%%%%%%%%%%%%%%%%%%%%%%%%%%%%%%%%%%%%%%%%%%%%%%%%%%%%%%%%%%%%%%%%%%%%

\section[Some non-squares modulo Fermat primes]{Some non-squares modulo all Fermat primes greater than $3$}\label{sec fermat primes}

We start with an easy lemma. 

%%%%
\begin{lemma}\label{lemnsquare}
For all $n\geq1$, we have
$$
2^{2^n}+1\equiv
\begin{cases}
3 \pmod 7 &\text{ if } n \text{ is even}\\
5 \pmod 7, &\text{ if } n \text{ is odd},
\end{cases}
$$
and 
$$
2^{2^n}+1\equiv 2 \pmod 3.
$$
\end{lemma}
\begin{proof}
Since $2^n$ is congruent to $(-1)^n$ modulo $3$, we have $2^n=1+3k$ for some odd $k$ when $n$ is even, and $2^n=2+3k$ for some even $k$ when $n$ is odd. Therefore, we have
$$
2^{2^n}=
\begin{cases}
2^{1+3k}=2\cdot 8^k\equiv 2\pmod 7&\textrm{ if $n$ is even}\\
2^{2+3k}=4\cdot 8^k\equiv 4\pmod 7&\textrm{ if $n$ is odd}.
\end{cases}
$$
and
$$
2^{2^n}=
\begin{cases}
2^{1+3k}\equiv 2\cdot (-1)^k\equiv 1\pmod 3&\textrm{ if $n$ is even}\\
2^{2+3k}\equiv 4\cdot (-1)^k\equiv 1\pmod 3&\textrm{ if $n$ is odd}.
\end{cases}
$$
\end{proof}
%%%%

%%%%
\begin{proposition}\label{nsquare}
The numbers 3 and 7 are not squares modulo all Fermat primes greater than $3$.
\end{proposition}
\begin{proof}
Let $p=2^{2^n}+1$ be a Fermat prime greater than $3$. By the quadratic reciprocity law, since $p\ne7$, we have
$$
\begin{aligned}
\binom{7}{p} \binom{p}{7} &=(-1)^{\frac{2^{2^n}\cdot 6}{4}}\\
                          &=(-1)^{2^{2^n-1}\cdot 3}\\
													&=1, 
\end{aligned}
$$
hence, $7$ is a square modulo $p$ if and only if $p$ is a square modulo $7$.  Since the set of squares modulo $7$ is $\{0,1,2,4\}$, we deduce by Lemma \ref{lemnsquare} that $p$ is not a square modulo $7$. 

Similarly, we have $\binom{3}{p} \binom{p}{3}=1$, so we can proceed as above. 
\end{proof}
%%%%

%%%%%%%%%%%%%%%%%%%%%%%%%%%%%%%%%%%%%%%%%%%%%%%%%%%%%%%%%%%%%%%%%%%%%%%%%%%%%%%%%%%%%%%%%%%%%%%%%%%%%%%%%%%
\section{Galois Group of $K_n$.}\label{sec galois group}

In this section, we assume that $\nu$ is not a square. 

Let $C_2$ be the cyclic group of order $2$, and denote by $[C_2]^n$ the $n$-fold wreath product of $C_2$ --- for basic facts about the wreath product, we refer the reader to \cite{Rotman}.  

Let $L_n$ be the Galois closure of $K_n$, and $\Gal(L_n)$ be its Galois group. The following is a particular case of a theorem by M. Stoll \cite[Section 3, p. 243]{Stoll}. 

%%%%
\begin{theorem}\label{thmWreathProdGal}
If $\nu$ is a multiple of $4$, then $\Gal(L_n)\cong [C_2]^n$.
\end{theorem}
%%%%

In order to show that $\sqrt{2}$ is not in $K_n$, we will show that it is not in $L_n$. For this we will use a counting argument. First we will show that there are exactly $2^n-1$ quadratic subfields of $L_n$. Then we will construct $2^n-1$ quadratic subfields, none of which is $\Q(\sqrt2)$. 

%%%%
\begin{lemma}\label{thLemGroups}
There are $2^{n}-1$ subfields of $L_n$ which are quadratic extensions of $\Q$. 
\end{lemma}
\begin{proof}
We will give two different proofs. By the Galois correspondence, we need to count how many subgroups $H$ of $[C_2]^n$ are such that the quotient $[C_2]^n/H$ is isomorphic to $C_2$. 

\emph{Proof 1}. We prove that $[C_2]^n$ has $2^n-1$ subgroups of index $2$ (they are maximal subgroups). Let $M$ be the set of maximal subgroups of $[C_2]^n$. Since $[C_2]^n$ has order $2^{2^n-1}$, it is a $2$-group. The groups in $M$ have index $2$, so they are normal. The intersection of all the maximal subgroups of $[C_2]^n$ is called the \emph{Frattini subgroup} of $[C_2]^n$ and is denoted by $\phi([C_2]^n)$. By \cite[Th 5.48]{Rotman}, the group $\phi=\phi([C_2]^n)$ is normal, and the quotient $[C_2]^n/\phi$ is an $\F_2$-vector space. Let $d$ be the dimension of this vector space. 

For every $H\in M$, since $\phi\le H \le [C_2]^n$, the quotient $H/\phi$ is a subspace of $[C_2]^n/\phi$, and every subspace of $[C_2]^n/\phi$ corresponds to a maximal subgroup $H$. It is easy to see that the number of non-trivial subspaces of a vector space of dimension $d$ over $\F_2$ is $2^d-1$. So we have $2^d-1$ maximal subgroups. 

On the other hand, by Burnside's basis Theorem \cite[Th 5.50]{Rotman}, all minimal systems of generators of $[C_2]^n$ have the same cardinal, and this cardinal is $d$. However, the cardinal of a minimal set of generators for wreath products of cyclic groups has been computed by Woryna --- see the comments after Theorem 1.1 in \cite{Wor}. In our case, we get $d=n$. 

\emph{Proof 2}. We use the following well-known results from Group theory. Let $G$ be a group and $D(G)$ the commutator subgroup of $G$. Let $H$ be any subgroup of $G$. The following are true: 
\begin{enumerate}
\item $D(G)$ is contained in $H$ if and only if $H$ is a normal subgroup of $G$ and $G/H$ is abelian. 
\item If $D(G)$ is contained in $H$, then $(G/D(G))/(H/D(G))$ is isomorphic to $G/H$. 
\end{enumerate} 
Moreover, we need the fact that the quotient $[C_2]^n/D([C_2]^n)$ is isomorphic to $C_2^n$ --- see \cite[Proof of Lemma 1.5]{Stoll}.

Suppose that $H$ is a subgroup of $[C_2]^n$ with $[C_2]^n/H$ isomorphic to $C_2$. By item 1 above, we deduce that $D([C_2]^n)$ is contained in $H$. By item 2 and by Lagrange theorem, we have 
$$
|H/D([C_2]^n)|=\frac{2^n}{2}=2^{n-1}.
$$
Furthermore, the subgroups containing $D([C_2]^n)$ correspond bijectively to subgroups of $[C_2]^n/D([C_2]^n)$. As in the first proof, the group $C_2^n$ is a vector space of dimension $n$ over $\F_2$, and every subgroup of order $2^{n-1}$ corresponds bijectively to a subspace of dimension $n-1$. This number is well known to be $2^n-1$. 
\end{proof}
%%%%

For $n\ge q$, let $c_n=P_n(0)$ be the constant term of the minimal polynomial of $x_n$, and let $c_1=\nu=-P_1(0)$. 

%%%%
\begin{lemma}\label{lemStoll2}
Let $p$ be a prime that divides some $c_n$. Let $m=\min\{n \ge 1 \colon p\textrm{ divides } c_n\}$ and $e$ be the order of $c_m$ at $p$. For every $n$, $p$ divides $c_n$ if and only if $p^e$ divides $c_n$ if and only if $m$ divides $n$.
\end{lemma}
\begin{proof}
There is also a simple proof in \cite[proof of Lemma 1.1]{Stoll}, inspired by Odoni \cite{Odoni}. We give a very elemental proof for the sake of completeness. 

Recall that $P_n(t)=f^{\circ n}(t)$, where $f(t)=t^2-\nu$. If $n=\ell + m$ for some integer $\ell>0$, then we have 
$$
c_n=f^{\circ \ell}(f^{\circ m}(0))=f^{\circ \ell}(c_m)\equiv c_\ell \pmod{c_m^2},
$$
hence, if $n$ is a multiple of $m$, then $c_n$ has the same order at $p$ as $c_m$. Conversely, write $n=qm+r$, with $0\le r < m$. We have 
$$
c_n=f^{\circ r}(f^{\circ qm}(0)) \equiv c_r \pmod{c_{qm}^2}, 
$$
hence $c_n$ is congruent to $c_r$ modulo $p$. So, if $p$ divides $c_n$, then it divides $c_r$ with $r<m$, which is a contradiction unless $r=0$. 
\end{proof}
%%%%

We recall that non-zero rational numbers $a_1$, \dots, $a_n$ are \emph{$2$-independent} if their residue classes in the $\F_2$-vector space $\Q^*/(\Q^*)^2$ are linearly independent. In \cite[Section 1, p. 16]{Stoll}, Stoll proves the following theorem. 

%%%%
\begin{theorem}\label{thmStoll1}
The group $\Gal(L_n)$ is isomorphic to $[C_2]^n$ if and only if $c_1$, \dots, $c_n$ are $2$-independent. 
\end{theorem}
%%%%

We also need the following simple observation: $\sqrt{c_1}$, \dots, $\sqrt{c_n}$ all lie in $L_n$. 

%%%%%
%\begin{lemma}\label{lemStoll1}
%For each $n\ge1$, there is a natural number $u$ such that $u$ divides properly $c_n$, and for every $k$ such that $1\le k<n$, $u$ is coprime with $c_k$. Furthermore, $u$ is not a square. 
%\end{lemma}
%%%
%\begin{proof}
%We follow Stoll's arguement \cite{Stoll}. The $u$ that we are looking for is his $b_n$ --- see \cite[Lemma 1.1]{Stoll}. Stoll proves that $u$ does not divide $c_k$ for $k<n$ and, for our choice of $\nu$, $u$ is not a square in $\Z$ --- see the proof of his Theorem in \cite[Section 3]{Stoll}. 
%\end{proof}
%%%%%

We can now prove our theorem.

%%%%
\begin{theorem}\label{thGalois}
Suppose the $\nu=2^{2m}\mu$, with $\mu\ge3$ odd and square-free and $m\geq 1$.  The field $L=\bigcup_nL_n$ does not contain $\sqrt{2}$ --- so in particular $\Ocal^{(\nu,0)}$ does not contain $\sqrt{2}$.
\end{theorem}

\begin{proof} 
From Theorem \ref{thmWreathProdGal} and Theorem \ref{thmStoll1} the number $c_1$, \dots, $c_n$ are $2$-independent. There are 
$$
\binom{n}{1}+\dots+\binom{n}{n}=2^n-1
$$
distinct possible products $\sqrt{c_{i_1}}\dots \sqrt{c_{i_k}}$. By the observation above, each product corresponds to a distinct quadratic extensions in $L_n$. We conclude with Lemma \ref{thLemGroups} that there are no more. 

Since $c_1=\nu$, by Lemma \ref{lemStoll2}, $2^{2m}$ is the highest power of $2$ which divides $c_n$ for each $n\ge1$. Hence, in every product of the $\sqrt{c_i}$, an even power of $2$ comes out of the square root, and we deduce that $\sqrt2$ does not appear in any of the quadratic extensions that we found. 
\end{proof}

Here is an example. For $\nu=12$, we have $L_1=\Q(\sqrt{12})=\Q(\sqrt3)$, and 
$$
L_2=L_1\left(\sqrt{12+\sqrt{12}},\sqrt{12-\sqrt{12}}\right)=\Q\left(\sqrt3,\sqrt{12+\sqrt{12}},\sqrt{12-\sqrt{12}}\right).
$$
We have
$$
\sqrt{12+\sqrt{12}}\sqrt{12-\sqrt{12}}=\sqrt{12^2-12}=\sqrt{12\cdot11}=2\sqrt{33}=\sqrt{c_2}.
$$
Hence, in $L_2$, we have the three following square roots: $\sqrt{3}$, $\sqrt{33}$ and $\sqrt{11}$. 

%Here is an example. For $\nu=28$, we have $L_1=\Q(\sqrt{28})=\Q(\sqrt7)$, and 
%$$
%L_2=L_1\left(\sqrt{28+\sqrt{28}},\sqrt{28-\sqrt{28}}\right)=\Q\left(\sqrt7,\sqrt{28+\sqrt{28}},\sqrt{28-\sqrt{28}}\right).
%$$
%We have
%$$
%\sqrt{28+\sqrt{28}}\sqrt{28-\sqrt{28}}=\sqrt{28^2-28}=\sqrt{28\cdot27}=6\sqrt{21}=\sqrt{c_2}.
%$$
%Hence, in $L_2$, we have the three following square roots: $\sqrt{7}$, $\sqrt{21}$ and $\sqrt{3}$. %There is no more, and as it is well known, $\sqrt2$ does not lie in $\Q(\sqrt7,\sqrt{21},\sqrt3)$. 

It is still an open problem to characterize the $\nu$ for which $\Gal(L_n)$ is $[C_2]^n$ for every $n$. Note that for $\nu=3$, the above does not work since $\sqrt{2}$ appears immediately in the tower. Nevertheless, for $\nu=7$, $\sqrt2$ does not appear in the first levels of the tower. This leads to the following question. 

\begin{question}
Is $\Gal(L_n)$ equal to $[C_2]^n$ when $\nu=7$? 
\end{question}

\noindent Marianela Castillo\\
Universidad de Concepci\'on, Los \'Angeles, Chile\\
Escuela de Educaci\'on\\
Departamento de Ciencias B\'asicas\\
Email: mcastillo@udec.cl\\

\noindent Xavier Vidaux (corresponding author)\\
Universidad de Concepci\'on, Concepci\'on, Chile\\
Facultad de Ciencias F\'isicas y Matem\'aticas\\
Departamento de Matem\'atica\\
Casilla 160 C\\
Email: xvidaux@udec.cl\\

\noindent Carlos R. Videla\\
Mount Royal University, Calgary, Canada\\
Department of Mathematics and Computing\\
email: cvidela@mtroyal.ca

\end{document}